\documentclass[10pt]{siamltex}
\usepackage{amsmath}
\usepackage{amssymb}
\usepackage{stmaryrd}

\usepackage{mdframed}
\usepackage{arydshln}
\newcommand{\ds}{\displaystyle}
\usepackage{hyperref}
\newcommand{\R}{\mathbb{R}}
\hypersetup{
    colorlinks=true,
    linkcolor=magenta,
    filecolor=magenta,
    citecolor=magenta,
    urlcolor=magenta,
    pdftitle={Overleaf Example},
    }
    \newcommand{\tensorTwo}{}

\usepackage{caption}
\usepackage{subcaption}

\usepackage{chngcntr}

\usepackage{listings}
\usepackage{comment}
\usepackage{cleveref}
\usepackage{algorithm}
\usepackage{algpseudocode}
\usepackage{graphicx}
\usepackage{color}
\usepackage{bf1}
\usepackage[top=2.9cm,bottom=2.9cm,left=2.9cm,right=2.9cm]{geometry}
\newcommand{\K}{\mathcal{K}}

\title{Parallel Matrix-free polynomial preconditioners \\ with application to flow simulations \\ in discrete fracture networks}

\date{}
\author{L. Bergamaschi, M. Ferronato, G. Isotton, C. Janna and A. Mart\'{\i}nez}
\begin{document}
\maketitle
\makeatletter
\@date
\makeatother

	\begin{abstract}
		We develop a robust matrix-free, communication avoiding parallel, high-degree polynomial preconditioner for the Conjugate Gradient method
		for large and sparse symmetric positive definite linear systems.
                We discuss the selection of a scaling parameter aimed at avoiding unwanted clustering of eigenvalues of the preconditioned
		matrices at the extrema of the spectrum.
		We use this preconditioned framework to solve a $3 \times 3$ block system arising in the simulation of fluid flow in large-size discrete
		fractured networks. We apply our polynomial preconditioner to a suitable
		Schur complement related with this system, which can not be explicitly  computed
		because of its size and density. 
		Numerical results confirm the excellent properties of the proposed preconditioner up to very high polynomial
		degrees. The parallel implementation achieves satisfactory
		scalability by taking advantage from the reduced number of scalar products and hence
		of global communications.

	\end{abstract}
	\begin{keywords}
	polynomial preconditioner, Conjugate Gradient method, parallel computing, scalability   
	\end{keywords}


	\section{Introduction}

Discretized PDEs and constrained as well as unconstrained optimization problems
often require the repeated solution of 
large and sparse linear systems $\ds A \fx = \fb$, in which $A$ is symmetric positive definite (SPD).
For practical scientific and engineering applications, the use of parallel computers
is mandatory, due to the large size and resolution of the considered models.
The size of these systems can be of order $10^6 \div 10^9$ and this calls for the use of
iterative methods, equipped with ad-hoc preconditioners as accelerators.

When the problem size grows up to several millions of unknowns,
it is not possible to store the system matrix nor the preconditioner on a single
machine. Furthermore, it is necessary to take advantage of several distributed resources
to reduce simulation time and, ultimately, the time to market.
Also, in many cases the huge size of the matrices can prevent their complete storage. In these 
instances only the application of the matrix to a vector is available as a routine (\textit{matrix
-free regime}).
Differently from direct factorization methods, 
iterative methods do not need the explicit knowledge of the coefficient matrix, however
they need to be suitably preconditioned to produce convergence in a reasonable CPU time. The issue is the
construction of a preconditioner  $P \approx A^{-1}$ which also works in a matrix-free regime.
The most common (general-purpose) preconditioners, such as the incomplete Cholesky factorization or most of approximate 
inverse preconditioners, rely on the knowledge of the coefficients of the matrix. An exception
	is represented by the AINV preconditioner (\cite{MR1787297}), whose construction is however
	inherently sequential. In all cases 
	 factorization based methods are not easily parallelizable, the bottleneck being
	the solution of triangular systems needed when these preconditioners are applied to a vector inside a Krylov subspace based solver.

	In this paper we are concerned with the effective development of 
	polynomial preconditioners, i.e. preconditioners that can be expressed as $P \equiv p_k(A)$.
Polynomial preconditioners are almost ideal candidates to be used as matrix-free parallel preconditioners, since, both
in set-up and application, they rely solely on operations, such as the sparse matrix
by vector product (SpMV), that are generally provided by highly efficient parallel linear
algebra libraries such as PETSci~\cite{petsc-web-page}, Hypre~\cite{FalYan02}, etc.
For instance, the application of $p_k(A)$ requires $k$ matrix-vector products, without needing
                      the explicit knowledge of the coefficients of matrix $A$.
	Moreover,
		their virtual  construction requires only the computation of the coefficients of the polynomials,
			with negligible computational cost, and
		the eigenvectors of the preconditioned matrix are the same as those
			of $A$. This feature can help accelerating the effect of the polynomial preconditioners
			by low-rank updates, which take advantage from the (approximate) knowledge of the eigenvectors
			of  $P A$.

The use of polynomial preconditioners for accelerating Krylov subspace methods is not new.
We quote for instance the initial works in \cite{MR694525,doi:10.1137/0906059} 
and \cite{VANGIJZEN199591,LiuMorganWilcox2015} where polynomial preconditioners are used
		to accelerate the Conjugate Gradient and the GMRES \cite{SaSc86} methods, respectively.
However, these ideas have been recently resumed, mainly in the context
of nonsymmetric linear systems, e.g. in  \cite{loe2019new,LTB_SIAMPP}
	or in the  acceleration of the Arnoldi method for eigenproblems \cite{LoeMorgan}.
	An interesting contribution to this subject is \cite{kaporin} where Chebyshev-based polynomial
	preconditioners are applied in conjunction with sparse approximate inverses.

	In this paper, starting from the work in \cite{CMM}, we develop a modified Newton-Chebyshev polynomial preconditioner
	for SPD systems, based on the choice of a parameter aimed at avoiding clustering of eigenvalues around the
	extrema of the spectrum. A theoretical analysis drives the choice of this parameter.	
	This matrix free preconditioner is employed in the solution of 
	the discrete problem
	arising from 
    flow simulations in discrete fracture network (DFN) models.
DFN models represent only the fractures as intersecting planar polygons, neglecting the surrounding
underground rock formation. The explicit representation of the fractures and their properties in a 
fully 3D structure requires the prescription of
continuity constraints for the fluid flow along the linear intersections. The number of
the fractures and their different size, that can change of orders of magnitude, entail a complex
and multi-scale geometry, which is not trivial to address. The problem has been effectively
reformulated as a PDE-constrained optimization problem in \cite{Berrone-et-alA,Berrone19}. The formulation relies on
the use of non-conforming discretizations of the single fractures and on the minimization of a
functional to couple intersecting planes, with no match between the meshes of the fractures
and the traces. The problem, often characterized by a huge size, can be algebraically reduced to the solution
of a sequence of SPD systems, whose matrix, however, cannot be computed and stored
explicitly. Nevertheless, the granular nature of the problem, which can be inherently subdivided in
several local problems on the fractures with a moderate
exchange of data, is particularly suitable for a massive parallel implementation.

In this work we will consider the Preconditioned Conjugate Gradient (PCG) method as iterative
solver, accelerated by the modified Newton-Chebyshev polynomial preconditioner.	 
For the parallel implementation, we rely on the Chronos library~\cite{CHRONOS-webpage,IsoFriSpiJan21}, a linear algebra
package specifically designed for high performance computing. 
Chronos takes advantage of fine-grained parallelism through the use of openMP
directives allowing for the use of multiple threads on the same MPI rank. 
Thanks also to the reduction of global
communication required by the repeated scalar products in PCG, the parallel implementation
of polynomial preconditioning turns out to be highly efficient, as will be shown in the
numerical experiments.

The rest of the paper is organized as follows: in Section 2 we briefly review the Newton-Chebyshev polynomial
	preconditioner and develop a strategy to avoid unpleasant clustering of eigenvalues around the endpoints
	of the spectrum. In Section 3 we show how to use our polynomial preconditioner in combination
	with other accelerators. In Section 4 we describe the test case arising from the DFN application, as 
	well as its algebraic formulation after finite element discretization and reduction to an SPD linear system.
	In Section  5 we describe our parallel implementation, while Section 6 collect the numerical results
	of the testing. Section 7 provides some concluding remarks.

 \section{Polynomial preconditioners}
 \label{sec:polprec}
 We briefly review two alternative formulations of the optimal polynomial preconditioners
 for the Conjugate Gradient method for symmetric positive definite linear systems, following the work in
 \cite{CMM}.   The connection between an accelerated Newton method for the matrix equation $X^{-1} = A$ and the Chebyshev 
 polynomials has been first established in \cite{PanPol} to develop a formula for matrix inversion.

 \subsection{Newton-based preconditioners}
\label{Newton}
The Newton preconditioner can be obtained as a trivial application of the Newton-Raphson method to the scalar equation
\[ x^{-1} - a = 0, \quad a \ne 0,\]
which reads 
\[ x_{j+1} = 2 x_j - a x_j^2, \quad j =0,\ldots, \qquad x_0 \ \text{fixed}.\]
The matrix counterpart of this method applied to $P^{-1} - A = 0$ can be cast as
\begin{equation}
	\label{newtonP}
 P_{j+1} = 2 P_j - P_j A P_j, \quad j = 0, \ldots, \qquad P_0 \text{ fixed}, 
\end{equation}
which is a well-known iterative method for matrix inversion (also known as Hotelling's method
\cite{hotelling1943}).

The efficiency of such a Newton method can however be increased
due to the following result, whose elementary proof is in \cite{CMM}:
\begin{theorem}
	\label{newtTh}
	Let $\alpha_j, \beta_j$ be the smallest and the largest eigenvalues of $P_j A$.

	If $0 < \alpha_j <  1  < \beta_j \le 2 - \alpha_j$ then $[\alpha_{j+1}, \beta_{j+1}]
	\subset [2 \alpha_j - \alpha_j^2, 1]$.
\end{theorem}

\noindent
If $\beta_j = 2-\alpha_j$ then the reduction in the condition number from $P_j A$ to $P_{j+1}A$ is near 4
provided that $\alpha_j$ is small:
\[ \frac{\kappa(P_j A)}{\kappa(P_{j+1} A)} = \frac{2-\alpha_j}{\alpha_j} (2\alpha_j - \alpha_j^2)
= (2 - \alpha_j)^2 \approx 4.\]
Under these hypotheses each Newton step provides an average halving of the CG iterations (and hence
of the number of scalar products) as
opposed to twice the application of both the coefficient matrix and the initial preconditioner. 
This idea  can be efficiently employed setting e.g. $P_0 = I$ to cheaply obtain a polynomial preconditioner. Other
choices of $P_0$ will be shortly discussed in Section 3.

At the first Newton stage the preconditioner must be scaled by $\zeta_0 = \dfrac{2}{\alpha + \beta}$ in order to 
satisfy the hypotheses of Theorem \ref{newtTh}. Hence
the eigenvalues of 
$P_1 A = \left(2 \zeta_0 I - \zeta_0^2 A\right) A$ lie in $[\alpha_1, \beta_1] $,
where $\beta_1 = 1$ and $\alpha_1 = (2-\alpha \zeta_0) \alpha \zeta_0$, and the next scaling factor is
$\zeta_1 = \dfrac{2}{1 + \alpha_1}$.
\begin{algorithm}[b!]
	\begin{algorithmic}[1]
		\caption{Newton-based polynomial preconditioner of degree $2^{\texttt{nlev}-1}$}
		\label{NewtAlg}
		\State Approximate the extremal eigenvalues of $A$: $\alpha, \beta$.
		\State Set the number of Newton steps: \texttt{nlev}  
		\smallskip
		\State Set $\zeta_0 = \dfrac{2}{\alpha + \beta}, \quad  
		\zeta_1 = \dfrac{2}{1+ 2\alpha\zeta_{0}-(\alpha\zeta_{0})^2},  \quad
		\zeta_i = \dfrac{2}{1+ 2\zeta_{i-1}-\zeta_{i-1}^2}, \quad i = 2, \texttt{nlev}.$  
		\smallskip
    \State At each CG iteration apply $P_{\texttt{nlev}}$ to the residual vector $\fr$ through the following recursive procedure:
			\begin{eqnarray*}
				P_0 \fr & = & \zeta_0 \fr \nonumber \\
				P_{j+1} \fr &=&  \zeta_{j+1} \left(2 P_j \fr - P_j A P_j\fr\right) , \qquad
				j = \texttt{nlev} -1, \ldots, 0
			\end{eqnarray*} 
	\end{algorithmic}
\end{algorithm}
Analogously, at a generic step $j >1$, $\alpha_j = (2-\alpha_{j-1} \zeta_{j-1}) \alpha_{j-1} \zeta_{j-1}$ and
$\zeta_j = \dfrac{2}{\alpha_j + 1}$. Finally, exploiting the relation $\alpha_{j-1} \zeta_{j-1}
= 2 - \zeta_{j-1}$ we can write
\begin{equation}
	\label{zetaNewt}
 \zeta_j = \dfrac{2}{1+ \zeta_{j-1} (2-\zeta_{j-1})} = 
	    \dfrac{2}{1+ 2\zeta_{j-1}-\zeta_{j-1}^2} .
\end{equation}
Then the recurrence for the preconditioners is obtained from (\ref{newtonP}) by scaling $P_j$ with
$\zeta_j$ as
\begin{eqnarray}
	\label{newrec}
	P_{j+1} &=& 2 \zeta_j P_j - \zeta_j^2 P_j A P_j, \quad j = 0, \ldots, \qquad P_0 = I.
\end{eqnarray}
which can be slightly improved by setting $\hat P_j  = \zeta_j P_j$, thus obtaining
\begin{eqnarray}
	\label{newrec0}
				\hat P_{j+1} &=&  \zeta_{j+1} \left(2 \hat P_j  - \hat P_j A \hat P_j \right) , \qquad
				\hat P_0  =  \zeta_0 I \nonumber
\end{eqnarray}
Application of the polynomial preconditioner to a vector $\fr$ is
described in step 4. of Algorithm \ref{NewtAlg}.

\subsection{Chebyshev preconditioners}
A similar recurrence can be obtained by means of the shifted and scaled Chebyshev polynomial preconditioners.
More details
can be found in \cite{Saad03,MR2169217,CMM}.
\label{ChebSec}
After setting  \[\theta = \frac{\beta+\alpha}{2}, \quad 
	 \delta = \frac{\beta-\alpha}{2}, \quad \text{and} \quad \sigma = \frac{\theta}{\delta} \]
the optimal polynomial preconditioner satisfies the following recursion:
\begin{eqnarray}
	p_{-1}(x) & = &  0 \nonumber \\
	p_0(x) & = &  \frac{1}{\theta} \nonumber \\
	p_k(x) & = & \rho_k\left(2 \sigma \left(1 - \frac{x}{\theta}\right) p_{k-1}(x)
	- \rho_{k-1} p_{k-2}(x) + \frac{2}{\delta}\right), \qquad k \ge 1. \label{polycheb}
\end{eqnarray}
with \begin{equation}
        \label{rhocheb}
        \rho_k = \dfrac{1}{2 \sigma - \rho_{k-1}}, \ k \ge 1 \quad \text{and} \quad \rho_0 = \dfrac{1}{\sigma}.
\end{equation}
The application of the Chebyshev preconditioner of degree $m$,  $P_m = p_m(A)$ to a vector $\fr$, satisfies
a three term recurrence.  In fact, defining
$\fs_k = P_k \fr, \ k \ge 0$, using (\ref{polycheb}) and exploiting 
the definitions of $\delta, \sigma$ and $\theta$, we have
\begin{eqnarray*}
	\fs_0 &=& \frac 1 \theta \fr. \\
	\fs_1 &=& \rho_1\left(2 \sigma \left(1 - \frac{A}{\theta}\right) p_{0}(A) + \frac{2}{\delta}\right) \fr
		= \dfrac{2\rho_1}{\delta} \left(2\fr - \dfrac{A \fr}{\theta}\right) \\
	\fs_k  & = &
	\rho_k\left(2 \sigma \left(1 - \frac{A}{\theta}\right) p_{k-1}(A)\fr
	- \rho_{k-1} p_{k-2}(A)\fr + \frac{2}{\delta}\fr\right) \\&& = \rho_k\left(2 \sigma \left(1 - \frac{A}{\theta} \right )\fs_{k-1}
	- \rho_{k-1} \fs_{k-2} + \frac{2}{\delta}\fr\right) 
	= \rho_k\left(2\sigma \fs_{k-1} - \rho_{k-1}\fs_{k-2}+\dfrac{2}{\delta} \left(\fr - A \fs_{k-1}\right)\right), \quad k > 1.
\end{eqnarray*}
The practical implementation of $P_m \fr$ is described in Algorithm \ref{ChebAlg}.

\begin{algorithm}
	\begin{algorithmic} [1]
		\caption{Computation  of the preconditioned residual $\hat \fr = P_m \fr$ with Chebyshev preconditioner.}
		\label{ChebAlg}
		\State Compute $\rho_k, k = 1,\ldots, m_{\max}$ using (\ref{rhocheb})
		\State $\fx_{old}  = \fr/\theta$ \quad \hspace{12.7mm} (\textit{if {m} $=0$ exit with 
		$\hat \fr = \fx_{old}$})
		\smallskip

		\State $\fx = \dfrac{2\rho_1}{\delta} \left(2\fr - \dfrac{A \fr}{\theta}\right)$
		\hspace{2.7mm}
		(\textit{if {m} $=1$ exit with 
		$\hat \fr = \fx$})

		\For {$k =2: {m}$}
		\smallskip
\State		$\fz = \dfrac{2}{\delta} \left(\fr - A \fx\right)$
\smallskip

		\State		$\hat \fr = \rho_{k}\left(2\sigma \fx- \rho_{k-1}\fx_{old}+\fz\right)$
		\State          $ \fx_{old} = \fx; \ \fx = \hat \fr$
\EndFor


	\end{algorithmic} 
\end{algorithm}
\subsection{Relation between Newton and Chebyshev polynomials}
In \cite{PanPol,CMM} a relation is established between the two algorithms basically by writing
a different recursion involving Chebyshev polynomials taken from the relation
\begin{equation}
 T_{2k}(x) = 2 T_k^2(x) - 1.
	\label{T2k}
\end{equation}
The Newton-based polynomial preconditioner is then proved equal to the Chebyshev polynomial preconditioner
based on the recursion (\ref{T2k}). Only, in the Newton case, polynomials in the sequence have degrees
$k = 2^j-1, \ j = 0, \ldots$, while with the original Chebyshev algorithm every nonnegative integer can be used
as the degree of the polynomial.
\subsection{Avoiding eigenvalue clustering}
A drawback of the polynomial preconditioners is that clustering may arise in the extremal parts
of the eigenspectrum of the preconditioned matrix, thus limiting the acceleration of the Conjugate Gradient method.
In \cite{CMM} a modification of the basic algorithms is proposed in order to mitigate such an undesired occurrence.
In this Section we analyze more deeply
the effect of this modification.

Let us first consider the first step of the original Newton approach.
 The spectral interval $[\alpha, \beta]$ of $A$ is first scaled by $\dfrac{2}{\alpha + \beta}  = \dfrac 1 \theta$ obtaining
 $[\hat \alpha,  \hat \beta] = \left[ \dfrac{2 \alpha}{\alpha + \beta}, \dfrac{2 \beta}{\alpha + \beta} \right]$. 
Following the results of Theorem 2.1 with $f(t) = 2t - t^2$, the spectral interval of $P_1 A$ is $[f(\hat \alpha), 1]$, with a reduction
 of the condition number of about 4, as explained in Section 2.1. However,
 \textit{the extrema of the scaled spectral interval are both mapped onto the left endpoint}
  $f(\hat \alpha) = f(\hat \beta) = \dfrac{4 \alpha \beta}{(\alpha+\beta)^2}$  of $P_1 A$ thus originating
  a cluster around the smallest eigenvalue, which is in principle detrimental for the CG convergence. 

  To avoid this, in \cite{CMM} a scaling parameter $\xi$ is introduced  in order to 
  modify the definition of parameter
$\theta$ in the Chebyshev/Newton algorithms as
\begin{equation} \label{thetamod}  \bar \theta = \dfrac{\beta+\alpha}{2} \left(1 + \xi\right).\end{equation}
	The parameter  $\xi$ should be small enough to apply just a slight modification of the native Chebyshev/Newton algorithm.
	Multiplying the  original spectral interval $[\alpha,  \beta]$ by  $\bar \theta^{-1} = \eta\theta^{-1}$ with $\eta = \dfrac{1}{1 + \xi}$, we obtain
 \[[\hat \alpha_\eta, \hat \beta_\eta] \equiv \left[ \dfrac{2 \eta \alpha}{\alpha + \beta}, \dfrac{2 \eta \beta}{\alpha + \beta} \right]\]
 which will be now mapped by the function $f(t)$ onto $[\alpha_\eta^{(1)}, \beta_\eta^{(1)}] := [f(\hat \alpha_\eta), 1]$. 

 Let us denote  by
 $\kappa = \dfrac{1}{f(\hat \alpha)}$  and
 $\kappa_{\eta} = \dfrac{1}{f(\hat \alpha_\eta)}$ the condition numbers of the preconditioned matrix before and after the modification,
 respectively.
 We first prove that modification (\ref{thetamod}) provides a modest increment of the condition number of the preconditioned matrix at step 1, assuming
 $\xi$ sufficiently small.
 \begin{theorem}
 Let $\xi = O(\kappa^{-1})$, then
		 \begin{equation*}   \dfrac {\kappa_\eta}{\kappa} =  1+ \xi + O(\xi^2).\end{equation*}
 \end{theorem}
 { 
 \begin{proof}
	 First we have that 
	 \[1-\eta = \frac{\xi}{1+\xi} = \xi + O(\xi^2), \quad \text{  and  } \quad
	 \hat \alpha_\eta - \hat \alpha = \left(\eta -1\right)
	 \dfrac{2 \alpha}{\alpha + \beta} = (\eta -1)  O(\kappa^{-1}) = O(\xi^2) \]
	 then
		 \begin{eqnarray*}  \dfrac {\kappa_\eta}{\kappa} = \dfrac{f(\hat \alpha_\eta)}{f(\hat \alpha)}   &
			 = &  \dfrac{f(\hat \alpha)  + (\hat \alpha_\eta - \hat \alpha) f'(\hat \alpha)  -2(\hat \alpha_\eta - \hat \alpha)^2}{f(\hat \alpha)}  \\
	 &=& 1 + \dfrac{\dfrac{2 \alpha }{\alpha + \beta} (\eta -1) \left(2 - 2 \dfrac{2 \alpha}{\alpha + \beta}\right)+ O(\xi^4)}
			 {\dfrac{4 \alpha \beta}{(\alpha+\beta)^2} }  =  \\
			  &=& 1 + 
			 \frac{(\alpha+\beta)^2}  {4 \alpha \beta}
			  \dfrac{2 \alpha }{\alpha + \beta} (\eta -1) \dfrac{2 (\beta - \alpha)}{\alpha + \beta}
			 + O(\xi^3)= \\
			  & = & 1 +  (\eta -1) \dfrac{\beta - \alpha}{\beta} + O(\xi^3) \\
			  &=&  \eta +  \frac{1}{\kappa}(1-\eta) + O(\xi^3) 
			  =  \eta +  O(\xi^2) =
		  1+ \xi + O(\xi^2).
		 \end{eqnarray*}
 \end{proof}
 }
%

%

 \noindent
 Though the condition number $\kappa_\eta$ slightly increases with respect to $\kappa$,
 the favorable outcome is that now $f(\hat \alpha_\eta) \ne f(\hat \beta_\eta)$ with a consequent separation of the smallest eigenvalues.
 Moreover,  a number $k \ge 1$ of the smallest eigenvalues are mapped onto as many of the smallest eigenvalues of the preconditioned matrix.
 The next theorem states that the $k$ (with $k \ge 1$) smallest eigenvalues of the preconditioned matrix are the map (through the function $f$) of exactly
 the $k$ smallest eigenvalues of $A$. This also means that the largest eigenvalues of $A$  are no longer mapped onto the same smallest eigenvalues
 of $P_1 A$, as it holds without modification.
 \begin{theorem}
	 \label{theo:uncluster}
	 Let $\eta$ be such that $\hat \alpha_\eta + 2(1 - \eta) < 1$. Denoting by 
\begin{eqnarray*} \hat \alpha_\eta = && \lambda_1^{(0)} \le \lambda_2^{(0)} \le \ldots \le \lambda_n^{(0)} = \hat \beta_\eta,
	 \qquad \text{and} \\
 \hat \alpha_\eta^{(1)} =&& \lambda_1^{(1)} \le \lambda_2^{(1)} \le \ldots \le \lambda_n^{(1)} = \hat \beta_\eta^{(1)} \end{eqnarray*}
		 the eigenvalues of $A$ and $P_1 A$, respectively, 
	 and $k$ the integer satisfying $\lambda_k^{(0)} \le \hat \alpha_\eta + 2(1 - \eta) \le \lambda_{k+1}^{(0)}$ then
	 \[ \lambda_j^{(1)} = f(\lambda_j^{(0)}), \qquad j = 1, \ldots, k.\]
 \end{theorem}
 \begin{proof}
 Since $f(t) = f(2-t), \ \forall t \in \R$ we have
 \[ f(\hat \beta_\eta) = f(2 - \hat \beta_\eta) = f\left(2 \dfrac{\alpha + (1-\eta) \beta}{\alpha + \beta} \right)
	 = f(\hat \alpha_\eta + 2(1 - \eta)). \]
	 Taking into account that the
	 function $f$ is increasing in $[\hat \alpha_\eta,1]$ and decreasing in $[1, \hat \beta_\eta]$ we have
	 \[f(\lambda_1^{(0)}) \le \ldots \le f(\lambda_k^{(0)}) \le f(\hat \alpha_\eta + 2(1 - \eta))  \le
	 \max_{j \ge k+1} f(\lambda_j^{(0)}), 
	 \]
	 and the thesis follows.
 \end{proof}

The situation is depicted in Figure \ref{Fig:nonwith} where the clustering (unclustering) of the extremal eigenvalues is shown
for $\xi = 0$ ($\xi = 0.05$). In this example we have $\lambda_1^{(0)}  = 0.1, \ \lambda_2^{(0)}  = 0.14, \ \lambda_2^{(0)} = 0.18$.
All these three eigenvalues are less than $\hat \alpha_\eta + 2(1 - \eta) \approx 0.195$ and therefore they are mapped
onto the leftmost part of the spectrum (blue asterisks, left panel).
With $\xi=0$ the eigenvalues $\lambda_{n-2}^{(0)}=1.82$, $\lambda_{n-1}^{(0)}=1.86$, $\lambda_n^{(0)}=1.9$ are mapped onto the same eigenvalues $\lambda_1^{(1)}$,
$\lambda_2^{(1)}$, $\lambda_3^{(0)}$, thus creating a cluster on the leftmost part of the spectrum. By distinction, with $\xi=0.05$ this is no longer true (blue asterisk, right panel).

\begin{figure}[h!]
	\begin{center}
	\caption{Eigenvalues of $A$, green squares on the $x$-axis, and of $P_1 A$, blue stars on the $y$-axis. Original algorithm (left), modified  algorithm with $\xi = 0.05$ (right).}
	\label{Fig:nonwith}
		\vspace{-4mm}
	\includegraphics[width=7.8cm]{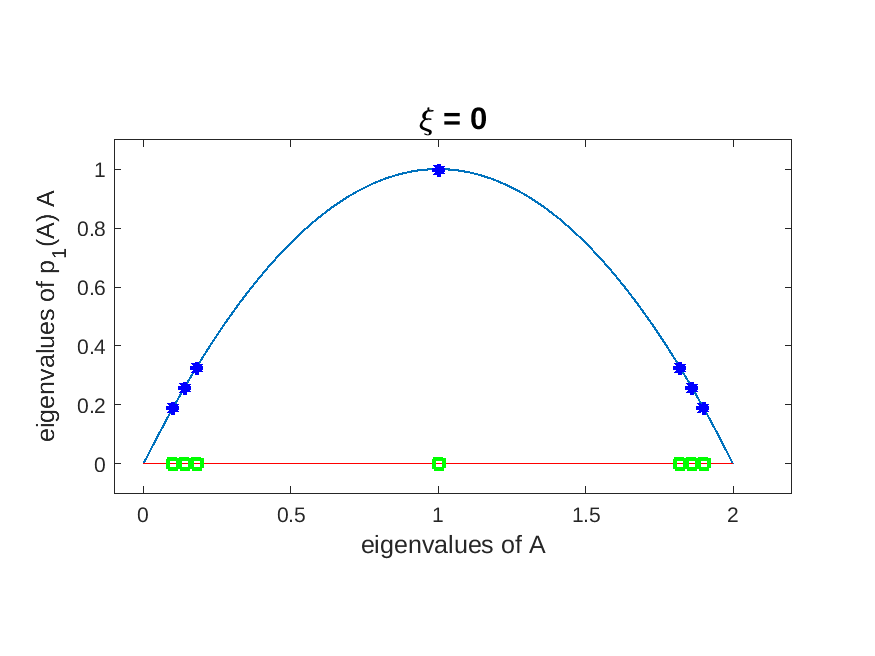}
		\hspace{-4mm}
	\includegraphics[width=7.8cm]{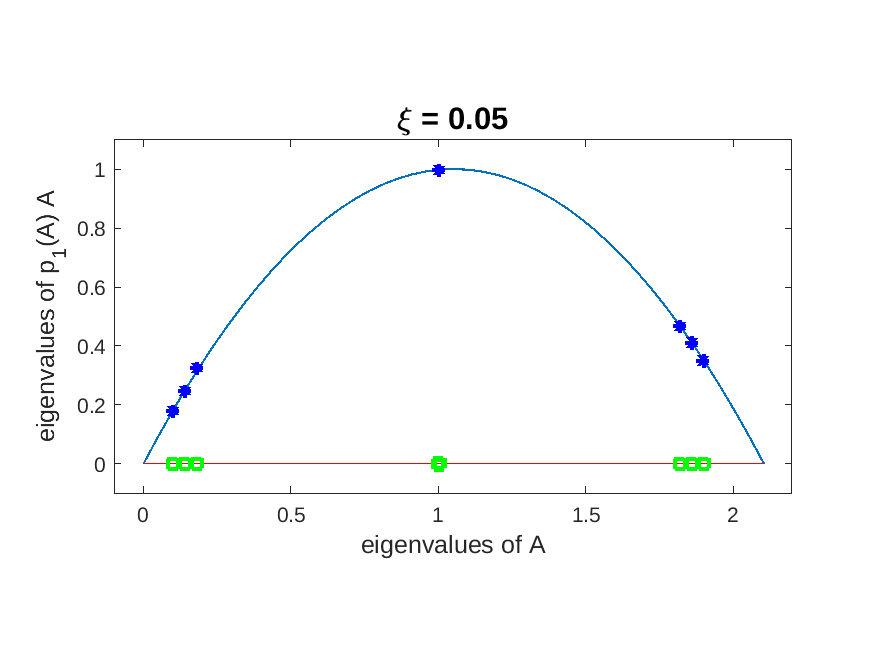}
	\end{center}
		\vspace{-5mm}
	\end{figure}

 Subsequent application of the Newton preconditioner will enhance this behavior: slight increase of the condition number (compared
 to the optimal one) at each Newton application, together with a progressive unclustering of the smallest eigenvalues. 
 To experimentally show this behavior we consider the solution of  the following linear system $A \fx = \fb$ with a random right hand side and  a diagonal matrix $A$ of size $n = 10^5$ such that
 \[ A_{ii} = i, \quad i = 1, \ldots, 10^5, \qquad  \text{nlev} = 6 \text{ (polynomial degree} = 63), \qquad \texttt{tol} = 10^{-10}.  \]
 We obtained the results summarized in Table \ref{tab:epsil} where we report the extremal eigenvalues of the preconditioned
 matrices for different values of $\xi$.  In addition to the condition number of $P_{63} A$ we computed a partial condition number, related
 to the $10$th smallest eigenvalue,
 $\kappa_{10} = \frac{\lambda_{\max}}{\lambda_{10}}$. 
 significantly.

 \begin{table}[h!]
	 \caption{PCG iterations to solve the diagonal problem and a few of the smallest eigenvalue
	 of the preconditioned matrices with
	 a polynomial preconditioner of degree $k = 63$, for different values of the scaling factor $\xi$. 
	 The condition numbers and the partial condition numbers are also provided.}
	 \label{tab:epsil}
 \begin{center}
 \begin{tabular}{l|rcccc|rr}
	 $\xi $ &  PCG iters & $\lambda_1$ & $\lambda_2$ &  $\lambda_5$ & $\lambda_{10}$  & $\kappa \equiv \frac{\lambda_{\max}}{\lambda_{1}}$ 
	 & $\frac{\lambda_{\max}}{\lambda_{10}}$ \\
	  \hline
	 $0$       &58& 0.03987 & 0.03987 & 0.03987 & 0.03987  & 25.08 &25.83\\ \hdashline[0.5pt/1pt]
	 $10^{-6}$ &57& 0.03984 & 0.04181 & 0.04181 & 0.04180  & 25.10 &23.92 \\
	 $10^{-5}$ &50& 0.03961 & 0.05901 & 0.05901 & 0.05901  & 25.25 &16.95\\
	 $10^{-4}$ &34& 0.03742 & 0.07388 & 0.17768 & 0.21046  & 26.72 & 4.75 \\
	 $10^{-3}$ &39& 0.02511 & 0.04976 & 0.12096 & 0.23084  & 39.82 & 4.33\\
	 $10^{-2}$ &62& 0.00898 & 0.01789 & 0.04419 & 0.08664  &111.31 &11.54  \\
 \end{tabular}                                                                               
 \end{center}
 \end{table}

 Obviously the smallest condition number is provided by the non modified algorithm ($\xi = 0$). 
 If $\xi$ is too small, then
 no significant effect is observed (second row in the Table). If $\xi$ is too large,  the unclustering of the eigenvalues
 does not pay for the large increasing of the condition number ($\xi = 10^{-2}$ in the Table).
 The optimal scaling is aimed at separating the smallest eigenvalues and at the same time reducing the partial condition number  $\kappa_{10}$
 (see last column in Table \ref{tab:epsil}) which is more informative about PCG convergence, when a few outliers (roughly 10 in this
 test case) are present~\cite{agSIAM}.

 The choice of the parameter $\xi$ is problem dependent. It is related to the degree of the polynomial, to the condition
 number of the original problem and to the separation of the smallest eigenvalues (to say nothing
 of the right-hand-side of the system).

\section{Polynomial acceleration of a given preconditioner}
Let us now assume that a (first level) preconditioner is available in factored form as 
\[P_{\text{seed}} = W W^T,\]
where $P_{\text{seed}}$ can be the square root of the inverse diagonal of $A$, the inverse of the Cholesky factor $W = L^{-1}$ or the triangular factor of an approximate inverse preconditioner.
In such a case the polynomial preconditioner can be applied to the symmetric matrix
\[\hat A = W^T A W.\]
If the  first level preconditioner can be constructed and applied in a matrix free environment then the whole
preconditioner 
can still be applied in a matrix-free environment.

\subsection{Low-rank acceleration}

The polynomial preconditioner needs the approximation of the two extremal eigenvalues, which are usually computed
together with the corresponding eigenvectors. In general, the availability of a number of the
leftmost (approximate) eigenvectors can be exploited to further improve the PCG convergence provided by the polynomial preconditioner.

Let us assume that $\fv_1, \ldots  \fv_p, \fv_{p+1}, \ldots, \fv_n$ are the eigenvectors of $A$ (or $\hat A$), and
$\lambda_1 \le \ldots \le \lambda_p \le \lambda_{p+1} \le \ldots \le \lambda_n$ the corresponding eigenvalues. Defining
\[ V = \begin{bmatrix} \fv_1 & \fv_2 & \ldots & \fv_p \end{bmatrix}, \qquad \Lambda = \text{diag} (\lambda_1, \ldots, \lambda_p), \]
	the polynomial preconditioner of degree $m$, $P_0$ in this section, computed for $A$ ($\hat A$) can be modified to obtain
	a spectral preconditioner as~\cite{CarDufGir,LB_Algorithms_2020}
			\[ P = P_0 + V (V^T A V)^{-1} V^T.\]
			Since $\fv_j, j = 1, \ldots, p$ are also eigenvectors of $P_0 A$,
			the following properties are easily verified:
			\begin{eqnarray}
				\label{plusone}
				P A \fv_j  &=& P_0 A \fv_j + \fv_j  =        (1 + p_m(\lambda_j)) \fv_j, \qquad \qquad j = 1, \ldots, p  \\
				P A \fv_j  &=& P_0 A \fv_j + 
				V (V^T A V)^{-1} \sum_{j=k}^p \tilde \fv_k^T \fv_j \approx
				p_m(\lambda_j) \fv_j, \qquad j = p+1, \ldots 
				\end{eqnarray}
				Since Theorem \ref{theo:uncluster} shows  that, with the $\xi$-modification,
                        the polynomial preconditioner matches the smallest eigenvalue of $A$
			on the smallest eigenvalue of $P_0 A$, 
			the latter are incremented by one, due to (\ref{plusone}), being
                        shifted in the interior of the spectrum with a consequent reduction of the condition number.

	\subsection{Preliminary Numerical Results}
	In this section we present some results in a sequential environment showing the acceleration provided by the polynomial preconditioner
	applied to a first level preconditioner and modified with low-rank matrices.
	We consider the solution of a linear system with matrix {\tt Cube\_5317k} (available at \url{http://www.dmsa.unipd.it/~janna/Matrices/}) arising from the equilibrium of a concrete cube discretized by a
regular unstructured tetrahedral grid with size $n = 5\,317\,443$ and nonzeros  \texttt{nnz} = $ 222\,615\,369 $.

\begin{table}[h!]
	\caption{Results for the matrix \texttt{Cube\_5317k}. The polynomial preconditioner has been modified with $\xi = 5\times 10^{-4}$}
	\label{tab:cubo}
	\begin{center}
	\begin{tabular}{r|rc|rc|rc|rc}
		&\multicolumn{4}{c|}{$P_{\text{seed}}= $ diagonal preconditioner} &
		\multicolumn{4}{c}{$P_{\text{seed}}= $ IC preconditioner} \\
		\hline
		& \multicolumn{2}{c|}{Polynomial + spectral } &
		 \multicolumn{2}{c|}{Polynomial }  
		& \multicolumn{2}{c|}{Polynomial + spectral } &
		 \multicolumn{2}{c}{Polynomial }  \\
	 deg & iter & CPU & iter & CPU & iter & CPU & iter & CPU \\
	\hline
		0   & 8597 &  4481.26 &9553&  4083.40&1359  &  1270.75 &1853  &  1476.07 \\
		1   & 4380 &  4038.17 &4865&  4066.76& 712  &  1283.02& 961   & 1499.21 \\
		3   & 2210 &  3829.54 &2434&  4008.73& 370  &  1219.82& 497   & 1599.90 \\
		7   & 1111 &  3726.72 &1226&  4006.06& 187  &  \textbf{1208.19}& 251   & 1594.50 \\
		15  &  563 &  \textbf{3716.18} & 620&  4042.62&  97  &  1240.66 &126   & 1645.59 \\
		31  &  292 &  3823.50 & 320&  4166.52&  51  &  1294.72 &64    &1605.11 \\
\hline              
\end{tabular}
	\end{center}
\end{table}

As the first level preconditioner we considered both the diagonal preconditioner and an incomplete Cholesky factorization 
with fill-in. In both cases we computed the 10 leftmost eigenpairs to a low accuracy ($\texttt{tol}=10^{-3}$ on the
relative residual). We neglect this preprocessing time taking in mind the case in which many linear systems
have to be solved with the same coefficient matrix (this is the case e.g. in linear transient problems).

			The sequential results provided throughout the paper have been obtained with a Matlab code running on an
Intel Core(TM) i7-8550U CPU \@1.80GHz.
The results  reported in Table \ref {tab:cubo} reveal that the combination of polynomial preconditioner and low-rank acceleration can be advantageous.

Considering for example the case with $P_{\text{seed}} = (\text{diag}(A))^{-1}$, the cost of the low-rank modification can be significant when the degree of the polynomial preconditioner is low while the relative influence of this task  decreases when the degree grows, since in this case the predominant cost is that of the high number of matrix-vector products.

\section{Example of application: Discrete Fracture Network (DFN) flow model}

As a relevant example of application of the proposed approach, we consider the DFN flow model developed in~\cite{Berrone-et-alA}.
The flow simulation in highly-fractured rock systems is computationally very demanding, because of the complexity of the
domain and the uncertainty characterizing the geometrical configuration. In this context, DFN models are usually preferred
when the fracture network has a dominant impact on the fluid flow dynamics. They explicitly represent the fractures as
intersecting planar polygons and neglect the surrounding rock formation, prescribing continuity constraints for the fluid
flow along the fracture intersections, usually called {\em traces}. Here, we briefly recall the original approach for
DFN models introduced in~\cite{Berrone-et-alA} and focus on its discrete algebraic formulation.

Let $\Omega$ be a connected three-dimensional fracture network consisting of the union of $n_f$ intersecting planar polygons
$\overline{\omega}_i$, $i=1,\ldots,n_f$, where $\overline{\omega}_i=\omega_i\cup\gamma_i$ is the closure of the open
planar domain $\omega_i$ with its linear boundary $\gamma_i$. The fluid flow through $\omega_i$ is assumed to be
laminar and governed by the standard mass balance equation coupled with Darcy's law, with appropriate essential and natural
boundary conditions on $\gamma_i$ to guarantee the well-posedness of the formulation:
\begin{subequations}
	\begin{align}
		-\nabla \cdot \left( \tensorTwo{\boldsymbol K} \nabla h \right) &= q, \qquad \mbox{in } \omega_i \in \Omega,
		\label{eq:flow_i} \\
		h_{|\gamma_i^D} &= h_i^D, \qquad \mbox{on } \gamma_i^D, \label{eq:Dir_i} \\
		\tensorTwo{\boldsymbol K} \nabla h \cdot \vec{n}_i &= g_i, \qquad \mbox{on } \gamma_i^N, \label{eq:Neu_i}
	\end{align}
	\label{eq:flow_model}
\end{subequations}
where $\gamma_i^D\cup\gamma_i^N=\gamma_i$, $\gamma_i^D\cap\gamma_i^N=\emptyset$, and $\gamma_i^D\neq\emptyset$.
In equations \eqref{eq:flow_model}, the scalar function $h$ is the hydraulic head, $\tensorTwo{\boldsymbol K}$ is the
fracture transmissibility tensor, which is assumed to be symmetric and uniformly positive definite, $\vec{n}_i$
is the outward normal to $\gamma_i^N$, $q$ is the known discharge within the fracture, and $h_i^D$ and $g_i$ are 
the given hydraulic head and flux prescribed along the fracture boundary, respectively. 
Since the fracture network is connected, there is a flux exchange through the linear traces between the intersecting
polygons. Let $\sigma_k^{i,j}$ denote the intersection between $\overline{\omega}_i$ and $\overline{\omega}_j$, which
we assume to be represented by a single close segment, with $\Sigma$ the union of the $n_s$ traces, $\Sigma=\cup_{k=1}^{n_s}
\sigma_k^{i,j}$. Indicating by $h_i$ the restriction of $h$ to $\overline{\omega}_i$, the continuity of the hydraulic 
head and the conservation of fluxes across the traces requires that:
\begin{subequations}
	\begin{align}
		h_{i|\sigma_k^{i,j}} - h_{j|\sigma_k^{i,j}} &= 0, \qquad \forall \; \sigma_k^{i,j} \in \Sigma,
		\label{eq:head_cont} \\
		\llbracket \tensorTwo{\boldsymbol K} \nabla h_i \cdot \vec{n}_k^i \rrbracket_{\sigma_k^{i,j}} +
		\llbracket \tensorTwo{\boldsymbol K} \nabla h_j \cdot \vec{n}_k^j \rrbracket_{\sigma_k^{i,j}} &= 0,
		\qquad \forall \; \sigma_k^{i,j} \in \Sigma, \label{eq:flux_cont}
	\end{align}
	\label{eq:trace_cont}
\end{subequations}
with $\vec{n}_k^i$ the outer normal to the trace $\sigma_k^{i,j}$ lying on the fracture $\overline{\omega}_i$ and 
the symbol $\llbracket \cdot \rrbracket_{\sigma_k^{i,j}}$ denoting the jump of the quantity within brackets through
$\sigma_k^{i,j}$. The DFN flow model consists of finding the hydraulic head $h:\Omega\rightarrow\R$ satisfying the
governing PDEs \eqref{eq:flow_model} under the constraints \eqref{eq:trace_cont}.

The numerical solution to the strong form \eqref{eq:flow_model}-\eqref{eq:trace_cont} is re-formulated 
in~\cite{Berrone-et-alA} as a PDE-constrained optimization problem in weak form. Let us introduce an appropriate
measurable function space $\mathcal{H}$ for the representation of $h$, such as, for instance:  
\begin{equation}
	\mathcal{H} = \left\{ \eta \in H^1 (\omega_i) : \eta_{|\gamma_i^D} = h_i^D, \forall i=1,\ldots,n_f \right\},
	\label{eq:Hspace}
\end{equation}
with $\mathcal{H}_0$ the corresponding counterpart with homogeneous conditions along $\gamma_i$. We use a mixed
formulation where the jump $\llbracket \tensorTwo{\boldsymbol K} \nabla h_i \cdot \vec{n}_k^i \rrbracket_{\sigma_k^{i,j}}$,
living along every trace $\sigma_k^{i,j}$ for all $i$ and $j$, is described by the unknown function $u_i:\sigma_k^{i,j}
\rightarrow\R$ belonging to the proper measurable function space $\mathcal{U}_i$, which is defined according to the selection
of $\mathcal{H}$. For example, for the choice \eqref{eq:Hspace}, $\mathcal{U}_i$ can be selected as a subspace of
$L^2(\sigma_k^{i,j})$, with the global space $\mathcal{U}$ including all $\mathcal{U}_i$. The set of constraints  
\eqref{eq:trace_cont} can be prescribed by minimizing the functional $\psi(h,u):\mathcal{H}\times\mathcal{U}\rightarrow\R$:
\begin{equation}
	\psi (h,u) = \frac 12 \sum_{\sigma_k^{i,j}\in\Sigma} \left( \left\| h_i - h_j \right\|^2_{\mathcal{H}} + \left\| u_i + u_j
	+ \alpha (h_i + h_j) \right\|^2_{\mathcal{U}} \right),
	\label{eq:func_psi}
\end{equation}
where $\alpha\in\R$ is a regularization parameter. The minimization of $\psi(h,u)$ under the conditions provided by 
equations \eqref{eq:flow_model} is enforced by using Lagrange multipliers. The weak form of \eqref{eq:flow_model} reads:
\begin{equation}
	\left( \nabla \eta, \tensorTwo{\boldsymbol K} \nabla h \right)_{\omega_i} - \left( \eta, u \right)_{\sigma_k^{i,j}} =
	- \left( \eta, q \right)_{\omega_i} + \left( \eta, g_i \right)_{\gamma_i^N}, \qquad \forall \; \eta \in \mathcal{H}_0,
	\; i=1,\ldots,n_f.
	\label{eq:weak_flow}
\end{equation}
Denoting by $p\in\mathcal{P}$ the Lagrange multipliers living in the appropriate space $\mathcal{P}$, the DFN flow solution
is obtained by finding $(h,u,p)\in\mathcal{H}\times\mathcal{U}\times\mathcal{P}$ that minimizes:
\begin{equation}
	\Psi (h,u,p) = \psi (h,u) + p \sum_i \left[ a_i ( \eta, h ) - c_i ( \eta, u ) - q_i (\eta) \right],
	\qquad \forall \; \eta \in \mathcal{H}_0,
	\label{eq:final_func}
\end{equation}
with $a_i(\eta,h)=(\nabla\eta,\tensorTwo{\boldsymbol K}\nabla h)$, $c_i=(\eta,u)_{\sigma_k^{i,j}}$, and $q_i=-(\eta,q)_{\omega_i}
+(\eta,g_i)_{\gamma_i^N}$.

\subsection{Discrete formulation}
The minimization of $\Psi(h,u,p)$ in \eqref{eq:final_func} is carried out approximately by replacing the function spaces
$\mathcal{H}$, $\mathcal{U}$ and $\mathcal{P}$ with their discrete counterparts $\mathcal{H}^h$, $\mathcal{U}^h$ and 
$\mathcal{P}^h$ with finite size $n^h$, $n^u$, and $n^p$, respectively. A relevant advantage of this formulation is
that independent computational grids can be introduced for each fracture following the standard finite element method,
with no need of enforcing the mesh conformity along the traces.

The  discrete counterpart of (\ref{eq:final_func}),
$\Psi(h^h,u^h,p^h)$, with $(h^h,u^h,p^h)\in\mathcal{H}^h,\mathcal{U}^h,\mathcal{P}^h$, is obtained 
by writing the three variables as linear combinations of the respective basis functions.
Denoting with $\fh = \begin{bmatrix}h_1, \ldots, h_{n^h}\end{bmatrix}^T$,
$\fu = \begin{bmatrix}u_1, \ldots, u_{n^u}\end{bmatrix}^T$ and
$\fp = \begin{bmatrix}p_1, \ldots, p_{n^p}\end{bmatrix}^T$ the vectors collecting the components
	of these linear combination we obtain the final expression of the discrete function to be minimized:
\[ \Psi(\fh, \fu, \fp) =  \begin{bmatrix} \fh & \fu \end{bmatrix}^T
                                \begin{bmatrix} G^h & -\alpha B \\
                                        -\alpha B^T & G^u \end{bmatrix}\begin{bmatrix} \fh \\ \fu \end{bmatrix}
						+\fp^T \left(A \fh  -C \fu -  \fq\right).  \]
The first order optimality conditions yield the following algebraic problem:
\begin{subequations}
\begin{align}
    G^h \fh - \alpha B \fu + A {\fp} &= \mathbf{0}, \label{eq:minJ_1} \\
    -\alpha B^T \fh + G^u \fu - C^T \fp &= \mathbf{0}, \label{eq:minJ_2} \\
	A \fh - C \fu &= \fq, \label{eq:mass_bal}
\end{align}
\label{eq:algebraic_model}
\end{subequations}

\noindent
where $\alpha$ is usually on the order of 1, $\fh\in\mathbb{R}^{n^h}$ is the discrete hydraulic head on fractures, 
$\fu\in\mathbb{R}^{n^u}$ is the discrete flux on the traces, and $\fp\in\mathbb{R}^{n^p}$ are the discrete
Lagrange multipliers. The vector $\fq\in\R^{n^h}$ includes the boundary conditions and the forcing terms.
Usually, $n^p=n^h$, while according to the problem $n^u$ can be either larger or smaller than $n^h$.
The matrices in \eqref{eq:algebraic_model} are as follows:
\begin{itemize}
	\item $G^h\in\mathbb{R}^{n^h\times n^h}$ and $G^u\in\mathbb{R}^{n^u\times n^u}$ are symmetric positive 
		semi-definite (SPSD), usually rank-deficient. The matrix $G^h$ is fracture-local, in the sense that it has 
		a block-diagonal structure with the block size depending on each fracture dimension, while $G^u$ has
		a global nature and operates on degrees of freedom related to different fractures;
	\item $B,C\in\mathbb{R}^{n^h\times n^u}$ are rectangular coupling blocks, whose entries are given by inner products 
		between the basis functions of $\mathcal{H}^h$ and $\mathcal{U}^h$. The matrix $C$ is fracture-local,
		with rectangular blocks whose size depends on the dimension of each fracture and the related traces,
		while $B=C+E$ has a global nature accounted for the contribution of matrix $E$ that has zero entries in the
		positions corresponding to the nonzero entries of the rectangular blocks of matrix $C$;
	\item $A\in\mathbb{R}^{n^h\times n^h}$ is symmetric positive definite (SPD) and fracture-local, i.e., with a 
		block diagonal structure. Each diagonal block arises from the discretization of the 
		$\nabla\cdot(\tensorTwo{\boldsymbol K}\nabla)$ operator over a fracture, hence inherits the usual structure 
		of a 2-D discrete Laplacian. 
\end{itemize}

Equations \ref{eq:algebraic_model} can be written in a compact form as:
\begin{equation}
    \left[ \begin{array}{ccc}
         G^h & -\alpha B & A \\
         -\alpha B^T & G^u & -C^T \\
         A & -C & 0
    \end{array} \right] \left[ \begin{array}{c}
         \fh \\
         \fu \\
         \fp
    \end{array} \right] = \left[ \begin{array}{c}
         \mathbf{0} \\
         \mathbf{0} \\
         \fq
    \end{array} \right]  \qquad \Longrightarrow \quad \K_0 \fx = \ff_0,
    \label{eq:sys}
\end{equation}
where $\K_0$ is a symmetric saddle-point matrix with a rank-deficient leading block. Solution to such problems arise 
in several applications and is the object of a significant number of works. For a review on methods and ideas, see 
for instance \cite{bgl05}. With an SPD leading block, as it often arises in Navier-Stokes equations, mixed finite 
element formulations of flow in porous media, poroelasticity, etc., an optimal preconditioner exists based on the 
approximation of the matrix Schur complement \cite{ESWbokk2005}. However, if the leading block is singular the 
problem is generally more difficult and the only available result is for the case of maximal rank deficiency \cite{EstGre15}.
					A potentially effective preconditioner for the system (\ref{eq:sys}) 
					has been recently proposed in \cite{GFBPS}
					where an appropriate permutation and inexact block factorization of $K_0$
					is obtained following the ideas developed in \cite{FFJCT} and \cite{FCF}. 
					The algorithm robustness, however, is problem-dependent and the overall solver may suffer from scalability issues.

\subsection{Algebraic solver and preconditioning strategy}
We develop here a preconditioning frame\-work exploiting the nice properties of matrix $A$, that is SPD, block diagonal, and 
such that its inverse can be applied exactly to a vector at a relatively low cost, and the polynomial acceleration. 
First, an appropriate permutation of $\K_0$ is used: 
\begin{equation}
        {\K} =\left [  \begin{array}{cc|c}
            A & 0 & -C \\
                G^h & A & -\alpha B \\[.1em]
		\cline{1-3}& \\[-.6em]
            -\alpha B^T & -C^T \quad & G^u
        \end{array} \right ],
        \qquad {\fx} = \begin{bmatrix}
         \fh \\
         \fp \\
         \fu
    \end{bmatrix}, \qquad {\ff} = \begin{bmatrix}
         \fq \\
         \mathbf{0} \\
         \mathbf{0}
    \end{bmatrix},
    \label{eq:tilde_var}
\end{equation}
so as to avoid a singular leading block. Though the permuted matrix is no longer symmetric, 
the $2 \times 2$ principal submatrix has a block diagonal structure and,  hence it is cheaply invertible.
In a more compact form, the permuted system $\K\fx=\ff$ can be written as
                        \begin{equation}
                                \label{block2}
\begin{bmatrix} M & -Z \\ -W^T & G^u \end{bmatrix}
\begin{bmatrix}\fx_1 \\ \fu \end{bmatrix} =
\begin{bmatrix}\ff_1 \\ \mathbf{0} \end{bmatrix}
                        \end{equation}
with
\[ M = \begin{bmatrix}   A & 0  \\ G^h & A \end{bmatrix}, \
   Z = \begin{bmatrix}   C \\ \alpha B \end{bmatrix}, \
   W = \begin{bmatrix}   \alpha B \\ C \end{bmatrix}, \
   \fx_1 = \begin{bmatrix}\fh \\ \fp \end{bmatrix}, \
	   \ff_1 = \begin{bmatrix}\fq \\ \mathbf{0} \end{bmatrix}. \]
Block Gaussian elimination reduces the system \eqref{block2} to:
\[ \begin{bmatrix} M & -Z \\ 0 &G^u - W^T M^{-1} Z \\ \end{bmatrix}
   \begin{bmatrix} \fx_1 \\ \fu \end{bmatrix} = 
   \begin{bmatrix} \ff_1 \\ W^T M^{-1}  \ff_1 \end{bmatrix} 
	   \qquad \text{with} \quad M^{-1} =
   \begin{bmatrix} A^{-1} & 0 \\ -A^{-1} G^h A^{-1} & A^{-1}\\ \end{bmatrix} \]
whose main computational burden is in the solution of
\begin{equation}
	\label{schur}
		S_u(\alpha) \fu = \fr, \qquad
                S_u(\alpha) = G^u - W^T M^{-1} Z, \quad \fr= W^T M^{-1}  \ff_1.
\end{equation}
Direct computation easily shows that matrix $S_u$ is symmetric:
\[ S_u(\alpha) = G^u - W^T M^{-1} Z = G^u  - \alpha B^T A^{-1} C - \alpha C^T A^{-1} B + C^T A^{-1} G_h A^{-1} C.\]
It is also positive definite under realistic conditions. In fact, matrix $S_u(\alpha)$ can always be made SPD by 
wisely selecting $\alpha > 0$ since $S_u(0)$ is SPD as the sum of the SPSD matrix $G_u$ and the SPD matrix 
$C^T A^{-1} G_h A^{-1} C$.
We assume that this assumption is verified and denote simply by $S_u$ the Schur complement in \eqref{schur}.
Therefore, the PCG solver can be employed.

Explicit computation of $S_u$ is not affordable for realistic problems, while the matrix-free application of $S_u$ 
to a vector can be implemented with no need of matrix-matrix multiplications. Before starting the PCG iteration, the
exact Cholesky factorization of $A$ is computed, i.e., the lower triangular matrix $L_A$ such that $A = L_A L_A^T $. 
Note that the Cholesky factor $L_A$ preserves the block diagonal structure of $A$ and each diagonal block arises from 
a 2-D discretization, hence this task is not overly expensive. Then, the application of $S_u$ to a vector $\fr$
can be implemented as described in Algorithm \ref{alg1},
whose arithmetic complexity amounts to 6 triangular solves plus 7 matrix-vector products involving block matrices
$B, C, G^h$ and $G^u$. Once system (\ref{schur}) is solved, the unknowns
$\fh$ and $\fp$ in (\ref{eq:tilde_var}) can be readily recovered by
\[ \fh = (L_A L_A^T)^{-1} \left(\fq + C \fu\right), \quad
\fp = (L_A L_A^T)^{-1} \left(B \fu - G^h \fh \right) .\]
The fact that the coefficient matrix $S_u$ is not explicitly available calls for a matrix-free preconditioner, namely the
Newton-Chebyshev polynomial preconditioner described in Section 3.

\begin{algorithm}[h!]
	\caption{Computation of $\fy = S_u \fr$}
        \label{alg1}
        \begin{algorithmic}[1]
		\State $\fv = C \fr$;
		\State $\fz = B \fr$;
		\State Solve $L_A\fu = \fv$;
		\State Solve $L_A^T\ft = \fu$;
		\State Solve $L_A\fu = \fz$;
		\State Solve $L_A^T\fw = \fu$;
		\State $\fz = G^u \fr - B^T \ft - C^T \fw$;
		\State $\fv = G^h \ft$;
		\State Solve $L_A\fu = \fv$;
		\State Solve $L_A^T\fw = \fu$;
		\State $\fy = \fz + C^T \fw$.
        \end{algorithmic}
\end{algorithm}

\subsection{Preconditioner implementation details}
		Following the discussion in Section 3,  we used as the \textit{seed} preconditioner the diagonal of $S^u$.
		Note that  $D_S = \text{diag}(S^u)$  can be computed without forming $S^u$ through the steps
		described in Algorithm \ref{algSu}, where with $\fz_i, \ft_i$ we denote the $i$-th column of
		matrices $Z$ and $T$, respectively.

		\smallskip

\begin{algorithm}[H]
	\caption{Computation of $D_S = \text{diag}(S^u)$}
	\label{algSu}
	\begin{algorithmic}[1]
	\State $Z  =  \left(L_A L_A^T\right)^{-1} C $
        \State $T  =  G^h Z - 2 B $
		\For {$i=1 : m$}
		\State ${(D_S)}_i  =   {(G^u)}_{ii} + \fz_i^T \ft_i$
		\EndFor
	\end{algorithmic}
\end{algorithm}

\noindent
The most time-consuming task in Algorithm \ref{algSu} is represented by the computation of $Z$ which requires
sparse matrix inversions. However, it must be observed
that these operations involve block matrices and hence do not produce a dramatic increase of the fill-in.

				 The polynomial preconditioner will be therefore applied to the symmetrically scaled system
					\[ \hat S^u  \hat{\fu} = \hat{\fr}, \qquad \text{with} \quad 
\hat S^u =\sqrt{D_S^{-1}} S^u  \sqrt{D_S^{-1}}, \ 
			\hat \fu=	\sqrt{D_S}\fu, \ 
			\hat \fr=	\sqrt{D_S^{-1}}\fr\]

\section{Parallel Implementation}

An efficient parallel implementation of the application of the Schur complement $S^u$ and
the explicit computation of its diagonal $D_S$ is fundamental for handling large-size problems
arising from realistic industrial applications.

The proposed algorithm is implemented relying on the Chronos software package,
a collection of linear algebra algorithms designed for high performance computers
\cite{CHRONOS-webpage}. Chronos is entirely written in C++ using the potential of
object-oriented programming (OOP) to easen its use from other software.
The Message Passing Interface (MPI) is used for communications among processes while
OpenMP directives enhance the fine-grained parallelism through multithreaded execution.
Chronos is free for research purposes and its license can be requested at the library
website~\cite{CHRONOS-webpage}.

The high level of abstraction introduced in Chronos by the OOP allows for the use of
the same distributed matrix object to store and use all the sparse matrices
composing the block system $\K$ in \cref{eq:tilde_var}.
In particular, Chronos adopts a Distributed Sparse Matrix (DSMat) storage scheme, 
where the matrix is sliced into $\texttt{nprocs}$ horizontal stripes of consecutive rows,
where $\texttt{nprocs}$ is the number of MPI ranks involved in the computation.
Each stripe is in turn subdivided into blocks stored in Compressed Sparse Row (CSR) format.
This block-nested storage scheme, along with nonblocking send/receive messages,
enhances the overlap between communications and computations hiding data-transfer latency
and reducing wall-time.

For the particular application of DFN, the stripes are chosen taking into account the
block-diagonal structure of the matrices $A$, $C$, and $G^h$. Each MPI rank stores a finite
number of consecutive blocks and no block is split between different ranks.
This subdivision then guides the partitioning of the other matrices $B$ and $G^u$.
A sketch of the DSMat storage scheme for the various blocks of the matrix $\mathcal{K}$ is shown in
Figure \ref{Fig:DSMat}.

\begin{figure}[h!]
\begin{center}
\caption{Chronos DSMat storage schemes for $A$, $C$ and $G^h$ (left) and $B$ and $G^u$ (center)
matrices partitioned into 4 MPI ranks. On the right, a corresponding distributed vector in
Chronos. The portions of matrices and vector stored by MPI rank 1 are highlighted with different
colors.}
\label{Fig:DSMat}
\includegraphics[width=12.0cm]{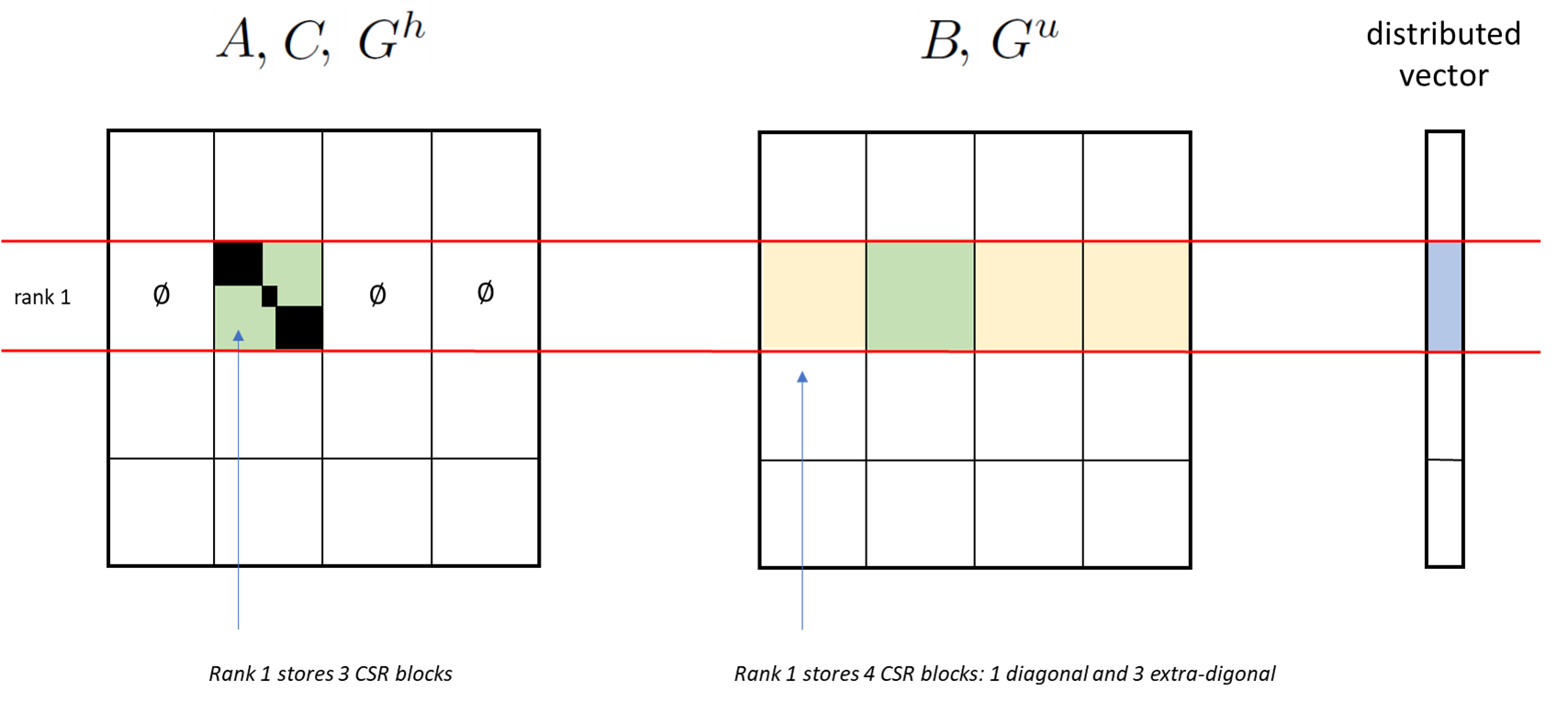}
\end{center}
\end{figure}

Both the multiplication by $S^u$ and the set-up of $D_S$ require the application of $A^{-1}$.
To this aim, the exact Cholesky factor of $A$ is computed by factorizing in parallel all
its diagonal blocks: since the number of blocks is very high, within each MPI rank,
several OpenMP threads are used to factor a chunk of blocks.
The sequential routine $cholmod\_factorize$, provided by the SuiteSparse library
\cite{SSPARSE-cholmod}, is used to factorize the single CSR blocks.

The $S^u$ application shown in Algorithm \ref{alg1} requires Sparse Matrix-by-Vector
product (SpMV) calls that are provided by Chronos. At its inner level, 
SpMV is specifically designed according to the type of matrix.
In particular, 10 SpMV products are executed with block-diagonal matrices, 6 of
which through forward and backward substitutions performed block-by-block using
$cholmod\_solve2$ from SuiteSparse. These products do not require any communication
between the MPI ranks, and on each rank the operations are executed by multiple OpenMP
threads. The remaining three SpMV products, involving $B$ and $G^u$, require preliminary
MPI data transfer: each stripe must receive the components of the distributed vector $\fr$
that correspond to the column indices of the extra diagonal CSR blocks.
To hide the latency, these communications are overlapped to the application of diagonal
CSR block with the portion of $r$ owned by the rank, highlighted respectively in green
and light blue in Figure \ref{Fig:DSMat}.

The computation of $D_S$ is performed in matrix-free setting following Algorithm \ref{algSu}.
Once again, the diagonal block structure allows for a highly parallel implementation that
does not require communications among MPI ranks.
In particular, each group of consecutive entries of $D_S$, corresponding to the rows
of a $C^T$ block, can be computed in parallel using  several OpenMP threads.

		\section{Numerical Results on the DFN problem}
		The relevant sizes and nonzeros of the test  matrices are reported in Table \ref{size}.
		\begin{table}[h!]
			\caption{Size and nonzeros of the relevant matrices for each test case.}
			\vspace{-2mm}
			\label{size}
			\begin{center}
				\begin{tabular}{rrrrrrrr}
			Test case		& $n^u$ & $n^p \equiv n^h$ & $nnz(\K)$ & $nnz(Z)$ &  $nnz(S^u)$ & \# fractures \\
			\hline
		\#1 & 56375  & 886693 & 13\,797084 & 301\,879683 & 62\,139981 & 395\\
		\#2	& $ 312518$ &  221144 & 10\,854803 &59\,966125& 325\,144680 &1425 \\
		\#3 (\texttt{Frac16}) &1\,428334 &502152 &31\,802122 & -- & -- & 15102 \\
		\#4 (\texttt{Frac32})&2\,777378 &994907 &44\,646710&--&--& 29370 \\
		\end{tabular}
			\end{center}
		\end{table}

		We notice that in the first case $n^u \ll n^h$ implying that the intermediate matrix $Z$ 
		has more nonzeros than the final Schur complement $S^u$, due to its large row size. For
		this problem it is more convenient  to form explicitly $S^u$ and work with the full Schur
		complement matrix.
In the other cases  computing the whole Schur complement is not worth due to its size and nonzero number,
so the computation of $\text{diag}(S^u)$ and the applications of $\hat S^u$ to a vector
are implemented as described in Algorithms \ref{alg1} and \ref{algSu}. The (very high) nonzero number
of $S^u$ for test case \#2 is reported only to reiterate that this matrix must not be formed explicitly.

\subsection{Results on test case \#1}
To roughly estimate the extremal eigenvalues we used the CG-based method called
Deflation-Accelerated Conjugate Gradient, DACG \cite{bgp97nlaa,BergamaschiPutti02} with
low accuracy, namely using a  tolerance on the relative residual
$\texttt{tol}_{\texttt{eig}} = 10^{-3}$. The DACG method is aimed at computing the leftmost
eigenpair of an SPD pencil $(A,B)$ but can be also employed to assess the (reciprocal of the)
largest eigenvalues of $A$ when the input matrices are $(I, A)$.
The DACG algorithm required 39 non preconditioned iterations for the smallest and 45 iterations
for the largest eigenvalue and 6.5 seconds overall.

\begin{table}[h!]
\caption{Iterations and CPU time to solve  $\hat{S} \hat{\fu} = \hat{\fr}$
         with the polynomial preconditioner for various degrees and $\xi = 10^{-3}$
        (left) and for different $\xi$-values with $m = 31$ and rank-one update (right).}
\label{tab11}
\begin{minipage}{11cm}
\begin{center}
\begin{tabular}{r|rrrr|rrrr}
$m$ & iter &   MVP      & ddot    & CPU 
& iter &   MVP      & ddot  & CPU  \\
\hline
0   & 1322 & 1322 & 3966 &105.59 &1235  & 1235 & 4900 &  99.21 \\
1   &  670 & 1340 & 2010 &100.95 & 625  &  1250 &2500 &  89.77 \\
3   &350& 1400    & 1050 &104.75 & 327  &1308   &1308 &  94.20\\
7   &177&1416     & 531  &105.66 & 166  &  1328&  664 &  95.02\\
15  &90& 1440     & 270  &108.09 &  85  & 1360   &340 &  97.62\\
31  &48 & 1536    & 144  &114.61 &  45  & 1440   &180 & 103.17\\
63  &28& 1792     & 84   &133.49 & 27  & 1728    &108 & 123.69\\
\hline & \multicolumn{4}{c|}{no update} & \multicolumn{4}{c}{rank-one update} \\
\end{tabular}
\end{center}
\end{minipage}
\hspace{2mm}
\begin{minipage}{4.0cm}
\begin{center}
\begin{tabular}{rr}
$\xi$            & iter \\ \hline
0                     & 63 \\
$10^{-4}$             & 51 \\
$10^{-3}$             & 45 \\
$3 \times 10^{-3}$    & 49 \\
$5 \times 10^{-3}$    & 53 \\
$10^{-2}$             & 61 \\
\\
\\
\end{tabular}
\end{center}
\end{minipage}
\end{table}

\begin{figure}[h!]
\vspace{-4mm}
\caption{PCG Convergence profiles for the DFN test case \#1 and different values of the 
         polynomial degree. Polynomial preconditioner with rank-one acceleration.}
\label{Fig:prof}
\centerline{\includegraphics[width=8cm]{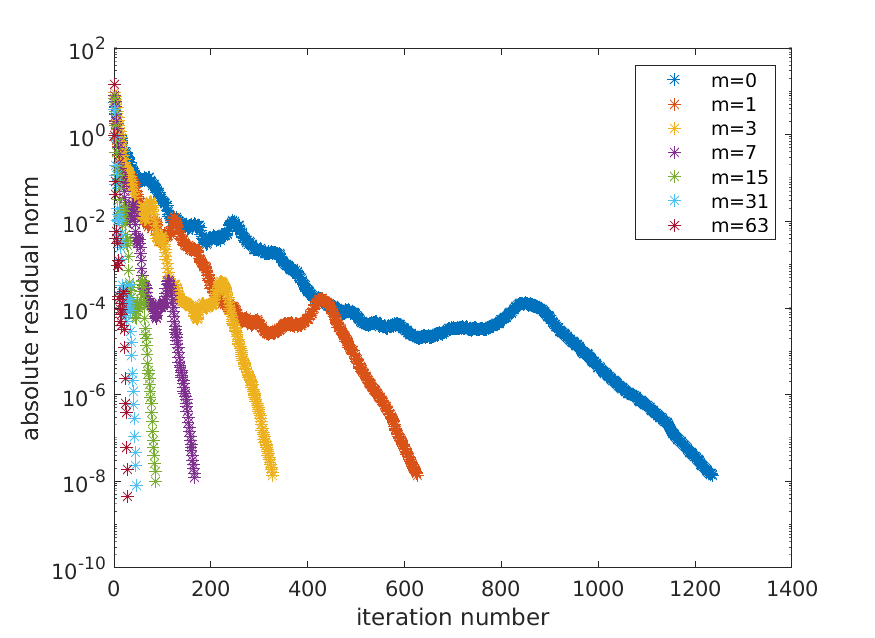}}
\end{figure}
The results in terms of number of iterations and CPU time are provided in Table \ref{tab11}
for increasing polynomial degree $m = 2^j-1, j = 0, \ldots, 6$. On the left
we show the results of the polynomial preconditioner alone, on the right with a
rank-one acceleration, namely using only the leftmost eigenpair, already computed for the
polynomial preconditioner setting.

The optimal scaling factor is found to be $\xi = 10^{-3}$ which is in accordance with the
theoretical findings as $\kappa(S^u) \approx 1.6 \times 10^4$.
The effect of the polynomial preconditioner is to drastically reduce the scalar products, by slightly
increasing the number of matrix-vector products. The low-rank correction, even using one vector only,
seems to be convenient, since the additional scalar product per iteration is compensated by a significant reduction
of the matrix vector products.
The convergence profile of the PCG solver with different polynomial preconditioners
is shown in Figure \ref{Fig:prof}, where the steepest profiles corresponding to
larger  degrees can be appreciated.

\subsection{Results on test case \#2} 
We use this test case to assess the parallel efficiency of our implementation of polynomial
preconditioning. We run the tests on the Marconi100 supercomputer which is installed at CINECA,
the Italian supercomputing center. Marconi100 consists of 980 computing nodes each one equipped
with 2 x 16 cores IBM Power9 AC922 processors at 2.6 GHz. For completeness, we add that each node
can also take advantage of 4 NVIDIA V100 GPU accelerators, but we do not use GPUs in these work.
The sparsity pattern of the whole $3 \times 3$ block matrix $\K$ is provided in Figure
\ref{Fig:spy}. Comparing this sparsity pattern with the block structure of $\K$ in equation
(\ref{eq:tilde_var}) we can observe that the nonzeros of the coupling matrices  $B$ and
$G^u$ are spread over the entire block while $A$, $C$ and $G^h$ display a block diagonal
structure. This is better shown in Figure \ref{spyA}, \ref{spyL} where a zoom of  matrix $A$ 
and its exact Cholesky factorization $L_A$ is provided.
\begin{figure}[h!]
\caption{Sparsity patterns of the whole matrix and subblocks.}
\begin{subfigure}[h]{0.32\textwidth}
\caption{Sparsity pattern of the $3 \times 3$ block matrix}
\label{Fig:spy}
\vspace{-1mm}
\includegraphics[width=4.6cm]{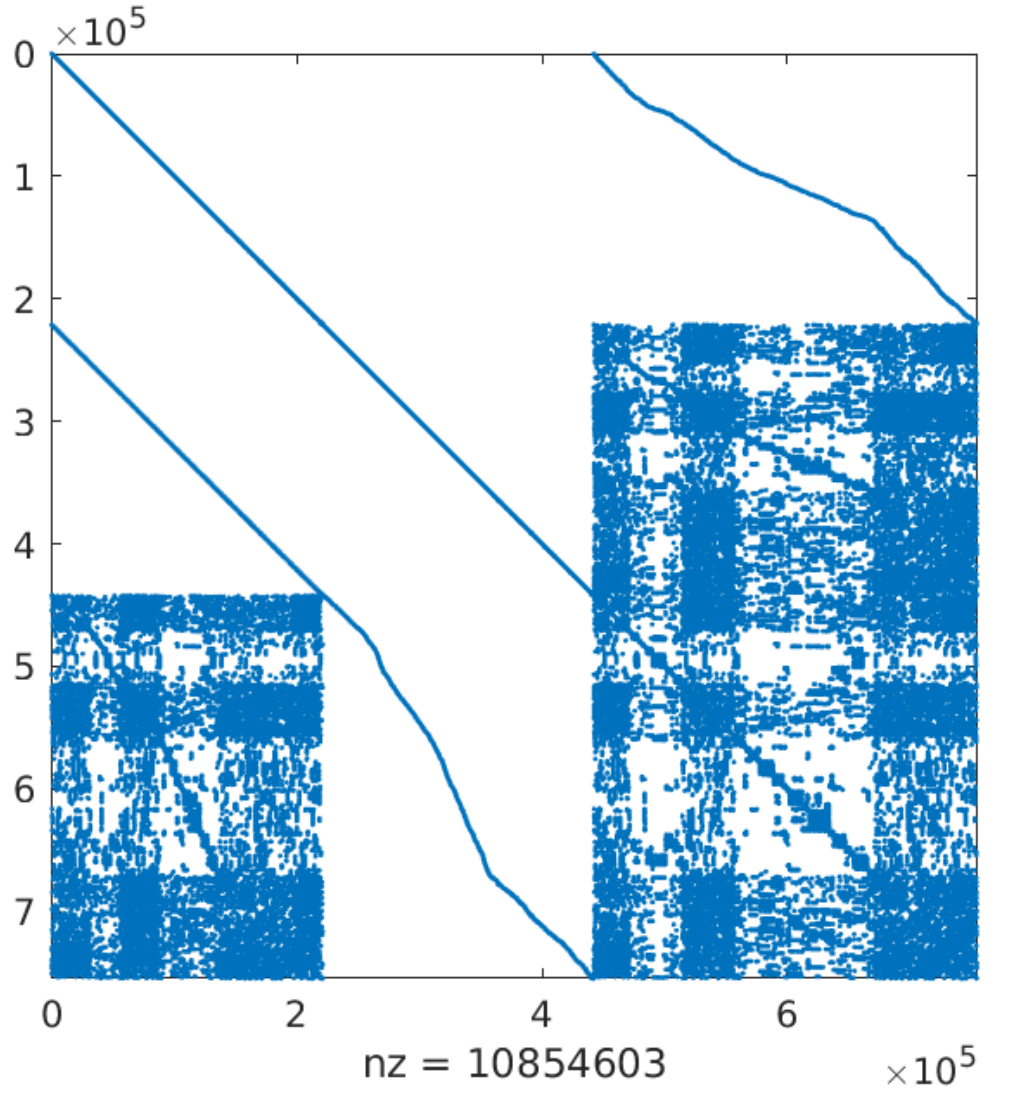}
\end{subfigure}
\hspace{-0.0cm}
\begin{subfigure}[h]{0.32\textwidth}
\caption{Zoom  of the leading submatrix of $A$ with size $n_0 = 1000$}
\label{spyA}
\includegraphics[width=5.0cm]{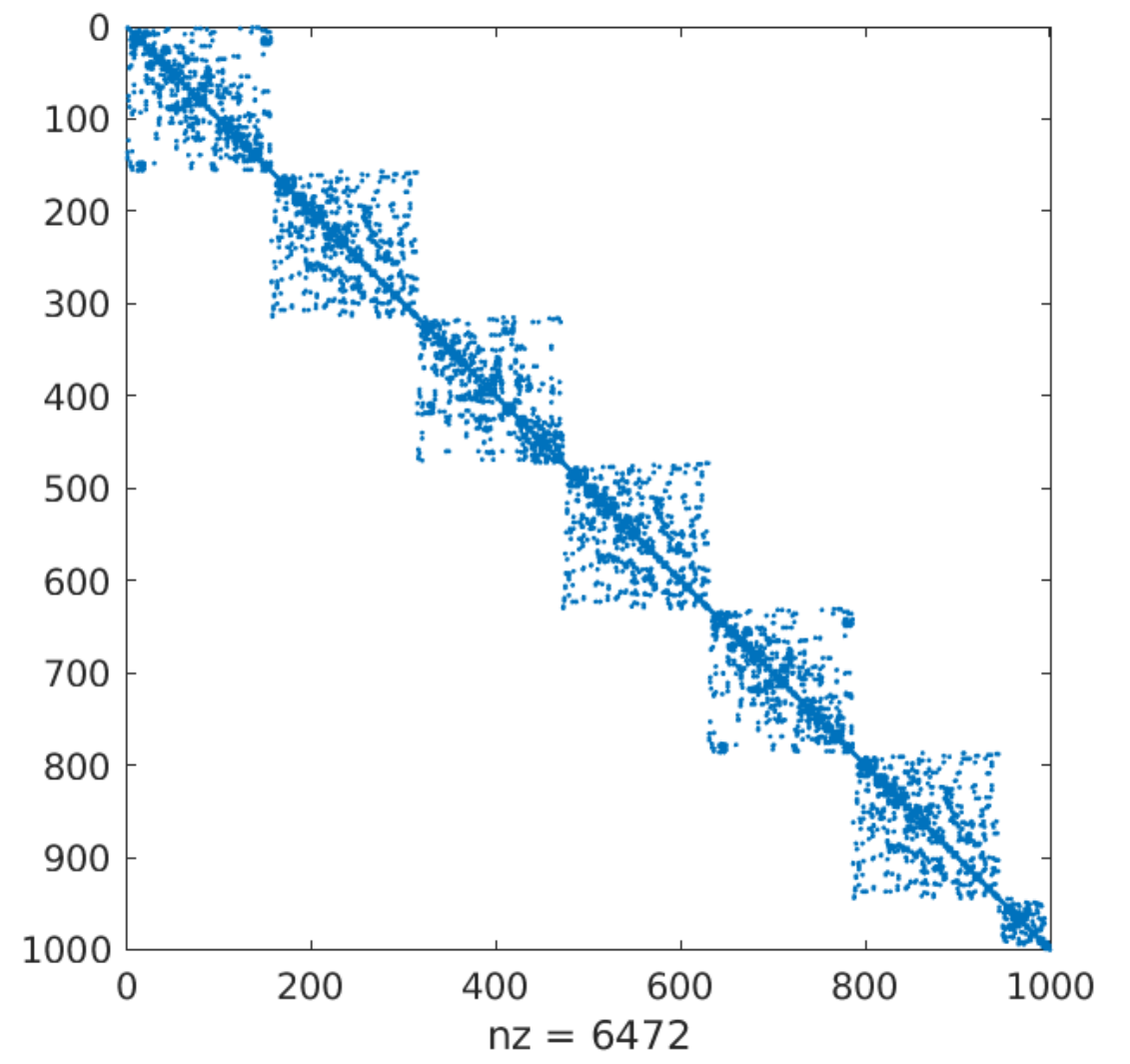}
\end{subfigure}
\hspace{-0.0cm}
\begin{subfigure}[h]{0.32\textwidth}
\caption{\small Zoom  of the leading submatrix of $L_A$ with size $n_0 = 1000$}
\vspace{2pt}
\label{spyL}
\includegraphics[width=5.0cm]{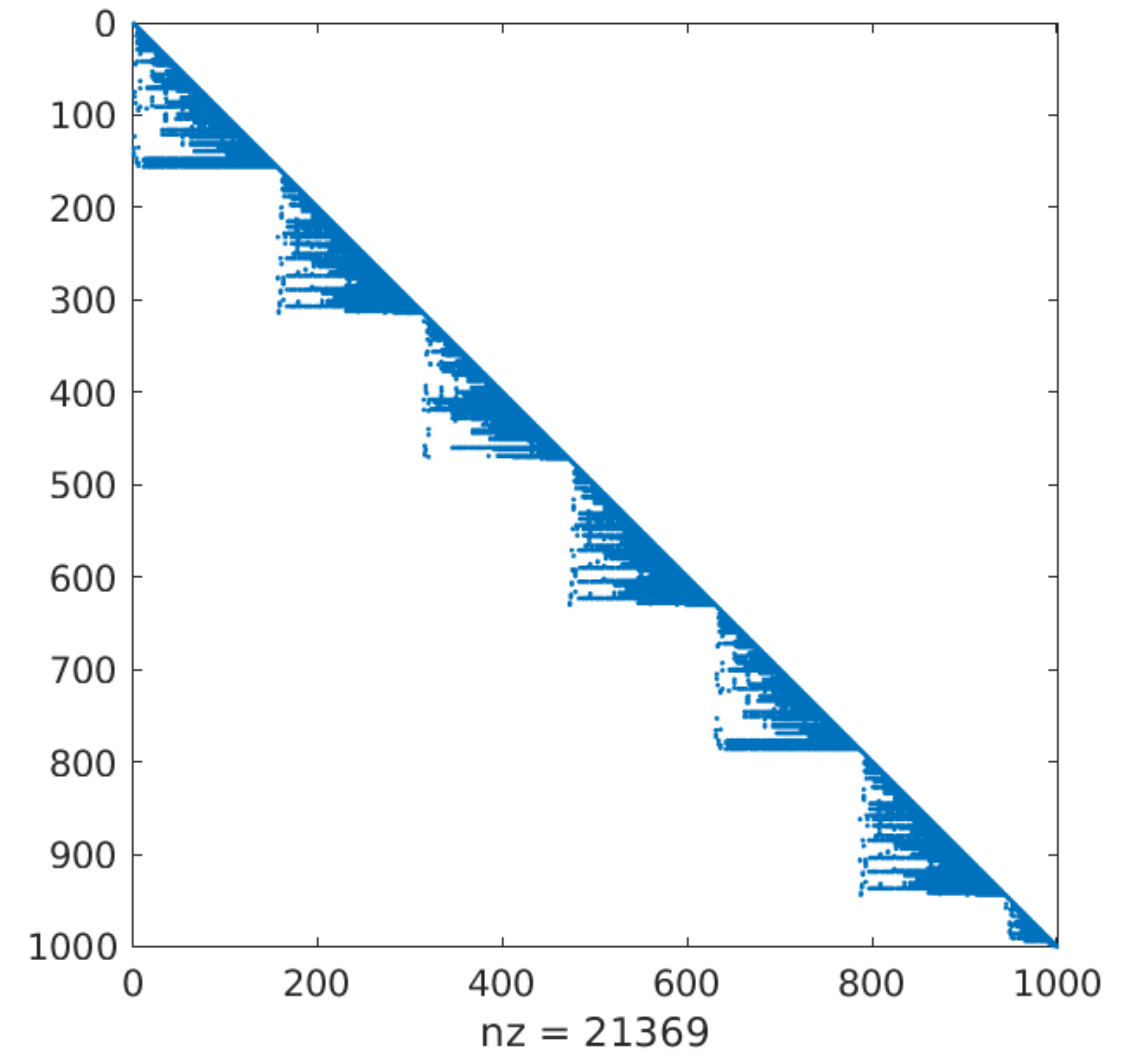}
\end{subfigure}
\end{figure}

Due to the large size of this problem, we solve it on 4 Marconi100 nodes and
involving all the available cores for a total of 128 cores. First, we experimentally
determine the optimal value of $\xi$ by varying it from 0.001 to 0.01 and keeping
fixed the polynomial degree to $m = 127$. Table \ref{tab:2xi} provides the number
of iterations to converge and solution time for PCG along with the minimum and maximum
eigenvalues of the diagonally scaled matrix that are needed to set-up the polynomial.

The choice of the polynomial degree has been made similarly by keeping $\xi = 0.007$ and varying
$m$, again on 128 cores of Marconi100. Table~\ref{tab:2m}, providing the number of
iterations to converge and solution time for PCG, shows that the number of iterations
always decreases with the degree of the polynomial, as expected, while the time to solution
initially decreases but reaches a minimum for $m = 127$.

\begin{table}[h!]
\caption{Number of iterations to converge and solution time for PCG preconditioned
with a polynomial of degree $m = 127$ and 128 Marconi100 cores by varying $\xi$
from 0.001 to 0.01. The minimum and maximum eigenvalues of the diagonally scaled matrix
are $1.56 \times 10^{-5}$ and $2.06$, respectively.}
\begin{center}
\begin{tabular}{rrr}
\hline
$\xi$ & PCG iters & Solv. time [s] \\
\hline
0.001 & 113 & 60.393 \\
0.002 & 107 & 56.562 \\
0.003 & 94 & 50.182 \\
0.004 & 83 & 44.074 \\
0.005 & 108 & 57.266 \\
0.006 & 97 & 51.527 \\
0.007 & 76 & 40.509 \\
0.008 & 78 & 41.758 \\
0.009 & 80 & 42.779 \\
0.010 & 83 & 43.959 \\
\hline
\end{tabular}
\end{center}
\label{tab:2xi}
\end{table}

\begin{table}[h!]
\caption{Number of iterations to converge and solution time for PCG preconditioned
with a polynomials of varying degrees and 128 Marconi100 cores for $\xi=0.007$.}
\begin{center}
\begin{tabular}{rrr}
\hline
$m$ & PCG iters & Solv. time [s] \\
\hline
3  &  2940  &  48.633 \\
7  &  1509  &  50.024 \\
15  &  670  &  44.493 \\
31  &  378  &  49.962 \\
63  &  195  &  51.636 \\
127  &  76  &  40.509 \\
255  &  46  &  49.445 \\
\hline
\end{tabular}
\end{center}
\label{tab:2m}
\end{table}

Finally, we provide a strong scalability test to demonstrate how polynomial preconditioning
is amenable to parallelization. Using the optimal values of $\xi$ and $m$ found above, that is
$0.007$ and $127$, respectively, we solve the test case \#2 by using 4 Marconi100 nodes and
a number of cores per node varying from 1 up to the maximum possible, 32. 
\begin{table}[h!]
\caption{Number of iterations to converge, solution time and parallel efficiency for
PCG preconditioned with a polynomials with a varying number of cores.}
\begin{center}
\begin{tabular}{rrrr}
\hline
\# of cores & PCG iters & Solv. time [s] & $\eta$ [\%] \\
\hline
4   &  76  &  552.0  &  100.00 \\
8   &  76  &  304.0  &   90.80 \\
16  &  76  &  175.3  &   78.75 \\
32  &  76  &  108.0  &   63.91 \\
64  &  76  &   63.8  &   54.05 \\
128 &  76  &   46.8  &   36.84 \\
\hline
\end{tabular}
\end{center}
\label{tab:eff}
\end{table}

From
Table~\ref{tab:eff}, it is possible to note how the number of PCG iterates remains constant,
as expected, while the solution times decreases with the increase of the number of cores.
To better understand how effective polynomial preconditioning is in parallel, we also report
the parallel efficiency which is defined as the ratio between real and ideal speed-up:
\begin{equation}
	\eta(\texttt{nprocs}) = \frac{\texttt{nprocs}}{4} \frac{T_{\texttt{nprocs}}}{T_4}
\end{equation}
where $\texttt{nprocs}$ denotes the number of cores used in the run and $T_{\texttt{nprocs}}$ the corresponding
execution time. Note that, although with 128 cores, the number of unknowns binded to each
core is only 2,441, we still have a reasonable efficiency which is very unlikely 
to reach with more complex preconditioning as approximate inverses, ILU or AMG.



\subsection{Results on the largest test cases}
This section presents the numerical results on the two largest test cases with a number of fractures of about 16,000 and 32,000, named \texttt{Frac16} and \texttt{Frac32}, respectively.
%
As done for the other test cases, we first determine the optimal value of $\xi$ by varying it from $10^{-4}$ to $5 \times 10^{-3}$ with a fixed polynomial degree $m$ = 127.
Table ~\ref{tab:Frac_epsilon} provides the number of iterations for the convergence of the PCG: the optimal value found is $10^{-4}$, but there are no significant differences in the range of $10^{-4} - 10^{-3}$. Moreover, the trend appears to be similar as the number of fractures increases.

\begin{table}[h!]
	\caption{Number of iterations for the convergence of the PCG preconditioned with a polynomial of degree $m$ = 127 by varying $\xi$ from $10^{-4}$ to $5 \times 10^{-3}$.}
\begin{center}
\begin{tabular}{rrrr|rrrr}
\hline
	Test case & $m$ & \multicolumn{1}{c}{$\xi$} & PCG iters &Test case & $m$ & \multicolumn{1}{c}{$\xi$} & PCG iters \\
\hline
&  127  &  $5 \times 10^{-3}$   & 132  & 
	    &  127  &  $5 \times 10^{-3}$  & 156 \\
	    \texttt{Frac16}  &  127  &  $1 \times 10^{-3}$   & 104  & 
	\texttt{Frac32}  &  127  &  $1 \times 10^{-3}$   & 121 \\
	&  127  &  $3 \times 10^{-4}$   & 105  & 
	    &  127  &  $3 \times 10^{-4}$   & 112 \\
	    &  127  &  $1 \times 10^{-4}$   & 103  & 
	    &  127  &  $1 \times 10^{-4}$   & 107 \\
 \hline
\end{tabular}
\end{center}
\label{tab:Frac_epsilon}
\end{table}

%
%
%
Regarding the parallel implementation, the two cases \texttt{Frac16} and \texttt{Frac32} were solved with degree $m$ = 127 and $\xi$ = 0.001 by increasing the number of cores up to 32. The results are provided in Table ~\ref{tab:Frac_cores} and show excellent strong scalability, with an efficiency of about 70\% with 32 cores where the number of unknowns binded to each core is only 15,000 and 30,000 for \texttt{Frac16} and \texttt{Frac32}, respectively.

\begin{table}[h!]
\caption{Number of iterations for the convergence, solution time and parallel efficiency of the PCG preconditioned
with polynomials of degree $m$ = 127 with a varying number of cores.}
\begin{center}
\begin{tabular}{rrrrr}
\hline
Test case & \# of cores & PCG iters & Solv. time [s] & $\eta$[\%] \\
\hline
         &  2  &  105 & 1678.6 & 100.0 \\
         &  4  &  104 &  866.6 &  96.8 \\
	\texttt{Frac16} &  8  &  104 &  459.6 &  91.3 \\
         & 16  &  103 &  249.7 &  84.0 \\
         & 32  &  103 &  157.5 &  66.7 \\
\hline
         &  4  &  107 & 1750.2 &  100.0 \\
	\texttt{Frac32} &  8  &  107 &  924.4 &   94.7 \\
         & 16  &  106 &  501.6 &   87.2 \\
         & 32  &  108 &  300.5 &   72.8 \\
\hline
\end{tabular}
\end{center}
\label{tab:Frac_cores}
\end{table}

\section{Conclusions}
A high-degree polynomial preconditioner has been developed with the aim of reducing the number of scalar products in the
Conjugate Gradient iteration. We have shown that the suitable choice of a scaling parameter can speed-up the PCG convergence
by avoiding clustering of eigenvalues around the endpoints of the spectral interval. We have given theoretical criteria to 
select an appropriate value for this parameter. The proposed preconditioning approach reveals particularly
useful when the coefficient matrix is not explicitly available, as in the case of the Schur complement matrix
obtained in the solution of a 3$ \times$ 3 block linear system arising in fluid flow simulations on fractured network models.
This preconditioner is well suited to parallelization since it reduces considerably the number of scalar product, 
thus minimizing the collective global communications among processors. 
Results on the Marconi100 supercomputer show satisfactory 
scalability results on realistic Discrete Fracture Networks test cases with thousands of fractures.

\bibliographystyle{siam}
				\bibliography{dfn}
\end{document}